\documentclass[letterpaper,10pt,conference]{ieeeconf}  

\IEEEoverridecommandlockouts                              

\usepackage{graphicx}
\usepackage{mathtools}
\usepackage{amsthm}
\usepackage{amsmath}
\usepackage{amssymb}
\usepackage[inline]{enumitem}
\usepackage{stmaryrd}
\usepackage{dsfont}
\usepackage{caption}
\usepackage{algorithm}
\usepackage[noend]{algpseudocode}
\usepackage{blindtext}
\usepackage[dvipsnames]{xcolor}
\usepackage{cite}

\newtheoremstyle{break}
{}
{}
{}
{}
{\bfseries}
{}
{ }
{\thmname{#1}\thmnumber{ #2}\thmnote{ (#3)}}
\theoremstyle{break}

\newtheorem{thm}{Theorem}
\newtheorem{lem}[thm]{Lemma}
\newtheorem{prop}[thm]{Proposition}

\newtheorem{rem}{Remark}

\DeclareMathOperator*{\argmin}{arg\,min}
\DeclareMathOperator*{\argmax}{arg\,max}

\title{\LARGE \bf
On Distributed Exact Sparse Linear Regression over Networks
}

\author{Tu Anh-Nguyen and C\'{e}sar A. Uribe
\thanks{TAN is with the Department of Computational and Applied Mathematics,
        Rice University, Texas, USA, {\tt\small tan5@rice.edu}. CAU is with the Department of Electrical and Computer Engineering, Rice University, Texas, USA,
        {\tt\small cauribe@rice.edu}.}%
}
\begin{document}

\maketitle
\thispagestyle{empty}
\pagestyle{empty}

\begin{abstract}
In this work, we propose an algorithm for solving exact sparse linear regression problems over a network in a distributed manner. 
Particularly, we consider the problem where data is stored among different computers or agents that seek to collaboratively find a common regressor with a specified sparsity~$k$, i.e., the $L_0$-norm is less than or equal to $k$.
Contrary to existing literature that uses $L_1$ regularization to approximate sparseness, we solve the problem with exact sparsity $k$. 
The main novelty in our proposal lies in showing a problem formulation with zero duality gap for which we adopt a dual approach to solve the problem in a decentralized way. This sets a foundational approach for the study of distributed optimization with explicit sparsity constraints.
We show theoretically and empirically that, under appropriate assumptions, where each agent solves smaller and local integer programming problems, all agents will eventually reach a consensus on the same sparse optimal regressor.

\end{abstract}

\section{Introduction}


Data regression analysis is a fundamental task in many modern research fields, ranging from natural sciences and engineering to management and social sciences~\cite{montgomery2021introduction}.
Linear regression is one of the most popular and well-studied methods to efficiently capture the relations between variables of interest and their predictors~\cite{weisberg2005}.
Analyzing the linear regressor is a common practice that yields meaningful interpretations of the data~\cite{seber2012}.
However, due to the high dimensionality of real-world data, such as RNA sequencing~\cite{moon2019}, it is a common practice to assume the linear regressor is sparse~\cite{tibshirani1996}.
A sparse regressor is not only computationally more efficient but also more interpretable compared to a dense solution~\cite{bertsimas2020sparse,zou2005regularization,jolliffe2003}.

Although sparse linear regression is a well-studied problem on a single machine~\cite{foster2016,bertsimas2020sparse}, it remains a challenge when dealing with modern distributed data storage.
Distributed data stored and data transfer between agents can be costly or access-controlled due to privacy policies. Thus, it becomes non-trivial to solve the sparse linear regression problem subject to strict communication and information constraints.

In this work, we propose a distributed sparse linear regression algorithm.
Formally, given $N$ agents, where each agent $i \in \llbracket N \rrbracket$ only has access to its local data $X^i \coloneqq (x^i_1, \dots, x^i_{n_i})^T \in \mathds{R}^{n_i \times p}$ and local observations $Y^i \coloneqq (y^i_1, \dots, y^i_{n_i}) \in \mathds{R}^{n_i}$. 
Ultimately, we want the group of agents to jointly solve the following optimization problem:
\begin{equation}
    \label{eq: sparse linear regression}
    \begin{split}
        \min_{w \in \mathds{R}^p} & \frac{1}{2} \sum_{i = 1}^N \|Y^i - X^iw\|^2_2 + \frac{1}{\gamma} \|w\|^2_2 \\
        \text{such that } & \|w\|_0 \leq k,
    \end{split}
\end{equation}
where $p$ is the dimension of the regressor.
Here, $\gamma > 0$ is a fixed parameter that controls the effect of the Tikhonov regularization term, and $k > 0$ is a predefined number which is interpreted as the number of non-zero coefficient of $w$ needed to model the data $\{(X^i, Y^i)\}_{i \in \llbracket N \rrbracket}$.
Since data transfer is prohibited in this setting, our proposed algorithm allows agents to exchange their intermediate parameters $w^i$ with their neighbors.
We show that the algorithm not only returns a sparse regressor but also guarantees a consensus across all agents.

A good amount of effort has been devoted to studying distributed linear regression over the last decades. 
In \cite{hellkvist2021linear, gascon2017privacy}, the authors study the problem where data are distributed vertically among units. In the setting where data is distributed among agents instead of features, Dobriban and Sheng \cite{dobriban2021distributed} study the scheme where each machine solves a linear regression problem locally and then sends the result to a central processing unit for averaging.
Mateos et al. \cite{mateos2010distributed} developed techniques for obtaining sparsity in distributed linear regression using Lasso.
In this work, instead of using Lasso to attain the desired sparsity approximately, we focus on solving \eqref{eq: sparse linear regression} exactly.
To the best of our knowledge, we are the first to consider such a problem in a distributed manner.

Most of the work done on sparse linear regression solves \eqref{eq: sparse linear regression} heuristically by replacing the combinatorial condition $\|w\|_0 \leq k$ by a $L_1$-norm constraint \cite{zou2005regularization}.
Elastic Net or Lasso is usually favored over solving \eqref{eq: sparse linear regression} exactly because of its computational feasibility and scalability.
However, they possess innate drawbacks as the $L_1$-constraint penalizes both large and small coefficients while, in contrast, the $L_0$-~constraint does not, and thus the sparsity pattern is not well recovered \cite{bertsimas2016best}.
Despite the NP-hardness of \eqref{eq: sparse linear regression}, Bertsimas et al. \cite{bertsimas2020sparse} have recently developed a cutting plane algorithm for solving \eqref{eq: sparse linear regression} in a matter of minutes where the number of data $n$ and the number of features $p$ is in order of $100,000$s.
With the ability to solve such a large combinatorial problem, we can compare the performance between Elastic Net and sparse regression. As shown in \cite{bertsimas2020sparse}, the solution of \eqref{eq: sparse linear regression} is superior in both accuracy and true support recovery.
Moreover, it has been shown both empirically and theoretically that the new cutting plane method requires much less data than Elastic Net to attain phase transition - the phenomenon in which the true support of regressors is recovered with enough data with high probability \cite{donoho2009observed, buhlmann2011statistics, david2017high}. 
Interestingly, in contradiction to the common intuition for the complexity of~\eqref{eq: sparse linear regression}, solving times for~\eqref{eq: sparse linear regression} drops significantly as the number of data increases~\cite{bertsimas2020sparse}.

The result in~\cite{bertsimas2020sparse} enables the computation of solutions of~\eqref{eq: sparse linear regression} in high-dimensional regimes. Moreover, sparse linear regression is a more promising candidate in the distributed setup than methods relying on $L_1$ regularization. Specifically:
\begin{enumerate*}
    \item Sparse linear regression tends to attain higher accuracy than $L_1$-based methods given the same amount of data, and the performance gap between these is larger when the number of samples is small. This is the case for distributed problems because each agent is not allowed to share data and thus can only process a limited number of data. 
    \item The phase transition of sparse linear regression occurs sooner than $L_1$-based methods. Thus sparse linear regression has a good chance to recover the support of the true regressor.
\end{enumerate*}

We employ a dual approach to solve exact sparse linear regression in a distributed manner. We first show that, even though the problem we want to solve involves binary variables, we can still achieve zero gap between the primal problem and the Lagrangian dual. 
A simple gradient ascent algorithm can converge to a sparse regressor.
However, each agent must solve a local quadratic integer program in the proposed dual framework at every iteration. As we shall see later, this problem is the sparse linear regression for the local data and observations at each agent plus an additional linear function.
To this end, we extend the outer approximation used in \cite{bertsimas2020sparse} for solving sparse regression to solve the local quadratic integer programming problem at every agent. This reformulation of the approximation method proposed in~\cite{bertsimas2020sparse} effectively makes the sparse linear regression problem \textit{dual-friendly} in the sense of~\cite[Definition 2]{uribe2021dual}

The rest of the paper is organized as follows: In Section \ref{section: algorithm and results}, we show that by using the distributed dual framework, we will converge to a sparse regressor.
Section \ref{section: local solver} provides an alternative transformation of the local quadratic integer programming problem within an agent so that we can solve it efficiently. In Section \ref{section: experiments}, we evaluate our distributed algorithm on a synthetic dataset and observe the convergence behavior with a different number of features, various network structures, and different network sizes.

\textbf{Notation:} We define $\llbracket n \rrbracket \coloneqq \{1, 2, \dots, n\}$.
Given a graph $\mathcal{G}$, we denote $V(\mathcal{G})$ and $E(\mathcal{G})$ as its vertices and edges set respectively.
For a node $i \in V(\mathcal{G})$,  we let $N(i) \coloneqq \{j \in V(\mathcal{G}) | (i,j) \in E(\mathcal{G})\}$ be its neighbor set.
For any set $S$, we use $\|S\|$ to represent the cardinality of $S$, and $\text{conv}(S)$ as its convex hull.
We denote $\mathds{R}_+ = \{x \in \mathds{R} | x \geq 0\}$.

\section{Algorithms and Results}
\label{section: algorithm and results}
We construct the distributed algorithm for exact sparse linear regression using the dual approach from \cite{uribe2021dual}.
First, we transform \eqref{eq: sparse linear regression} to a quadratic mixed-integer program using a big-M formulation \cite{vielma2015mixed}, which is a traditional technique to model $L_0$-constraints as linear inequalities with additional integer variables. 
Without the sparsity constraint $\|w\|_0 \leq k$, the optimization \eqref{eq: sparse linear regression} is a minimization of a strictly convex function, which has an unique minimizer $w^*$. 
Hence, there exists a real number $M \in \mathds{R}_+$ such that $\|w^*\|_2 \leq M$.
In practice, we do not need to compute the value of $M$, and we only use $M$ for the argument of the big-M formulation. In addition, we assume that there exists an underlying undirected graph $\mathcal{G}$ that represents which pair of agents can communicate with each other. 
The graph $\mathcal{G}$ is assumed to be connected, and we denote its Laplacian by $L \in \mathds{R}^{n \times n}$.
With these assumptions, \eqref{eq: sparse linear regression} can be rewritten as a quadratic integer programming (QIP) problem:
$$z_{QIP} = \min_{w^i \in \mathds{R}^p, \forall i \in \llbracket p \rrbracket} \sum_{i=1}^N \left(\frac{1}{2} \|Y^i - X^i w^i\|^2_2 + \frac{1}{\gamma N} \|w^i\|^2_2 \right)$$
\begin{subequations}
    \label{eq: qip}
    \begin{align}
        \label{cons: coupling}
        \text{such that   } & L v^j = 0 & \forall j \in \llbracket p \rrbracket &\\
        \label{cons: bounds}
        & -Ms^i \leq w^i \leq Ms^i & \forall i \in \llbracket N \rrbracket &\\
        \label{cons: sparsity}
        & s^i \in S^p_k & \forall i \in \llbracket N \rrbracket &,
    \end{align}
\end{subequations}
where $v^j = (w^1_j, \dots, w^N_j)^T$ denotes the vector consisting of the $j$-th entries of $w^1, \dots, w^N$, and $S^p_k = \{s \in \{0, 1\}^p | \mathds{1}^Ts \leq~k\}$.
In \eqref{eq: qip}, the coupling constraints \eqref{cons: coupling} assures that every agent has the same regressor.
Since the elements of the vector $s^i$ for every $i \in \llbracket N \rrbracket$ can only be $0$ or $1$, $w^i_j$ must be $0$ when $s^i_j = 0$ for some $j \in \llbracket p \rrbracket$, and takes an arbitrary value otherwise.
Thus, the two constraints \eqref{cons: bounds} and \eqref{cons: sparsity} enforces the sparsity on $w^i$.
The next result shows that we can apply Lagrangian multiplier theory for the coupling constraints and derive strong duality.

\begin{lem}
\label{lemma: zero gap}
Let $\gamma>0$, and the Lagrangian function of the  mixed-integer optimization Problem \eqref{eq: qip} be given by
\begin{align}\label{eq:phi}
    \phi(y) = \min_{\substack{s^i \in S^p_k, \\ -Ms_i \leq w^i \leq Ms_i \\ \forall i \in \llbracket N \rrbracket}} f(w, y),
\end{align}
where
\begin{align*}
    f(w, y) {\coloneqq} \sum_{i = 1}^N \left( \frac{1}{2} \|Y^i {-}X^i w^i\|^2_2{ +} \frac{1}{\bar{\gamma}} \|w^i\|^2_2 \right) {+} \sum_{j=1}^p \langle y^j, Lv^j \rangle,
\end{align*}
and the constant $\bar{\gamma} = \gamma N$. Then, $\max_y \phi(y) = z_{QIP}$, where $z_{QIP}$ is the optimal value of~\eqref{eq: qip}. 
\end{lem}
\begin{proof}
Let $\hat{y} = [\hat{y}^1, \dots, \hat{y}^p] \in \mathds{R}^{N \times p}$ be the maximizer of $ \phi(y)$. Then, we have
\begin{equation}
\notag
    \label{eq: phi y optimal}
    \begin{split}
        \phi(\hat{y}) = \min_{\substack{s^i \in S^p_k, \\ -Ms_i \leq w^i \leq Ms_i, \\ \forall i \in \llbracket N \rrbracket}} & f(w, \hat{y})
    \end{split}
\end{equation}
with $\hat{s}^1, \dots, \hat{s}^N$ and $\hat{w}^1, \dots, \hat{w}^N$ being its optimal solution. Thus,
\begin{align*}
       \phi(\hat{y}) &= \min_{\substack{s^i \in S^p_k, \\ -Ms_i \leq w^i \leq Ms_i, \\ \forall i \in \llbracket N \rrbracket}}  f(w, \hat{y}) \\
    & = \max_y \min_{\substack{-M\hat{s}^i \leq w^i \leq M\hat{s}^i, \\ \forall i \in \llbracket N \rrbracket}} f(w, y),
\end{align*}
where the second equality comes from the fact that $\hat{y}$ is the optimal solution of $\max_y \phi(y)$. 
Since $\hat{s}^i$ is fixed for every $i \in \llbracket N \rrbracket$, the set $\{(w^1, \dots, w^N)|~-M\hat{s}^i \leq w^i \leq M \hat{s}^i\}$ is a convex set. 
In addition, the objective function $f(w, y)$ is convex with respect to $w$ and linear (concave) with respect to $y$.  Hence, we have
\begin{align*}
        \phi(\hat{y}) & = \min_{\substack{-M\hat{s}^i \leq w^i \leq M\hat{s}^i, \\ \forall i \in \llbracket N \rrbracket}} \max_y  f(w, y) \\
        & \geq \min_{\substack{s^i \in S^p_k, \\ -Ms^i \leq w^i \leq Ms^i,\\ \forall i \in \llbracket N \rrbracket}} \max_y  f(w, y).
\end{align*}

Since $S^p_k$ is a finite set, the product set $(S^p_k)^N$ is also finite.
Thus, we can derive a lower bound on $\phi(\hat{y})$ as a minimum of a finite number of optimization problems as follows:
\begin{equation}
    \notag
    \phi(\hat{y}) \geq \min_{(s^1, \dots, s^N \in (S^p_k)^N)}   \min_{\substack{-Ms^i \leq w^i \leq Ms^i, \\ \forall i \in \llbracket N \rrbracket}} \max_y f(w,{y}) .
\end{equation}
Because for each $s^1, \dots, s^N \in (S^p_k)^N$, the domain $-Ms^i \leq w^i \leq Ms^i, \forall i \in \llbracket N \rrbracket$ satisfies the Slater condition \cite{slater2014lagrange}, we have that
$$\min_{-Ms^i \leq w^i \leq Ms^i, \forall i \in \llbracket N \rrbracket} \max_y f(w, y) = $$ $$ \min_{\substack{-Ms^i \leq w^i \leq Ms^i, \forall i \in \llbracket N \rrbracket, \\ Lv^j = 0 \ \forall j \in \llbracket p \rrbracket}} \sum_{i=1}^N \left( \frac{1}{2} \|Y^i - X^i w^i\|^2_2 + \frac{1}{\bar{\gamma}} \|w^i\|^2_2 \right),$$
and thus,
$$\phi(\hat{y}) \geq z_{QIP}.$$
However, since $\phi(y)$ is the Lagrangian function of \eqref{eq: qip}, we also have that
$$\phi(\hat{y}) \leq z_{QIP}.$$
Hence, we can conclude that $\max_y \phi(y) = z_{QIP}$ \\
\end{proof}
In the big-M formulation~\eqref{eq: qip}, we impose the consensus constraint by $Lv^j = 0, \forall j \in \llbracket p \rrbracket$. Therefore, we do not need to impose the sparsity condition on every agent.
Indeed, we can still derive the same result from Lemma \ref{lemma: zero gap} when we only require a subset of agents to solve the exact sparse linear regression. Lemma~\ref{lemma: zero gap} implies that even though~\eqref{eq: qip} is a quadratic integer programming problem, we still have zero gap between the primal objective and its Lagrangian dual. Hence, instead of minimizing~\eqref{eq: qip} with the hard coupling constraints, we can maximize $\phi(y)$.

Next, we state some useful properties of the function $\phi(y)$ that will enable us to propose a gradient ascent method for our dual problem.
\begin{prop}
\label{prop_grad}
The function $\phi(y)$ in~\eqref{eq:phi} is a concave and continuous function, whose gradient is given by
\begin{equation}
    \label{eq: dual gradient}
        \nabla \phi(y) = \begin{bmatrix}
        L \hat{v}^1 \\ L \hat{v}^2 \\  \vdots \\ L \hat{v}^p
        \end{bmatrix},
\end{equation}
where
\begin{equation}
\label{eq: for computing gradient}
\notag
        (\hat{v}^1, \dots, \hat{v}^p, s ) = \argmin_{\substack{s^i \in S^p_k, \\ -Ms^i \leq w^i \leq Ms^i, \\ \forall i \in \llbracket N \rrbracket}}  f(w, y)
\end{equation}
\end{prop}

Proposition~\ref{prop_grad} follows directly from Danskin's theorem~\cite{danskin1966theory}. 
Given that the function $\phi(y)$ is concave and has an explicit formulation for computing its gradient, we can find its maximum using the classical gradient ascent.

\begin{algorithm}[tb!]
\caption{Distributed Exact Sparse Linear Regression}
\label{algo: deslr}
\begin{algorithmic}[1]
\Procedure{Initialization}{}
    \State Each agent initializes its own local multipliers: 
    $$\psi^{i,1}_L \leftarrow (y^i_1, \dots, y^i_p) \ \forall i \in \llbracket N \rrbracket$$
\EndProcedure

\Procedure{Distributed Sparse Linear Regression}{$\{(X^i, Y^i)\}_{i \in \llbracket N \rrbracket}$, $G$, $k$, $\gamma$, $T$}
    \For{$t = 1, 2, \dots, T$}
        \State Agents receive multipliers from their neighbor
        $$\psi^{i,t}_N \leftarrow [\psi^{j, t}_L]_{j \in N(i)} \ \forall i \in \llbracket N \rrbracket$$
        \State Agents compute their dual multiplier
        $$D^{i, t} \leftarrow L_{i,i}\psi^{i,t}_L + \sum_{j \in N(i)}L_{i, j}\psi^{i, t}_{N, j}$$
        \State Agents solve their respective local problem:
        $$\min_{\|w^i\|_0 \leq k} \frac{1}{2}\|Y^i - X^i w^i\|^2_2 + \frac{1}{\bar{\gamma}} \|w^i\|^2_2 + \langle D^{i,t}, w^i \rangle$$
and obtaining local regressor $w^{i, t}$
        \State Agents send and receive new regressor from their neighbor, then update their local multiplier
        $$\psi^{i, t+1}_L \leftarrow \psi^{i,t}_L + \alpha_t (L_{i,i}w^{i,t} + \sum_{j \in N(i)}L_{i,j}w^{j,t})$$
    \EndFor
\EndProcedure
\end{algorithmic}
\end{algorithm}

We can derive the general framework for solving~\eqref{eq: qip} as described in Algorithm \ref{algo: deslr}.
In Algorithm~\ref{algo: deslr}, we use $\psi^{i, t}_L$ and $\psi^{i, t}_N$ to denote the local multiplier and neighbor multiplier of agent $i$ at iteration $t$ respectively.
We should keep in mind that the variable $\psi^{i,t}_L$ is just a rearrangement of the Lagrangian multiplier $y$ in the function $\phi$.

Based on the strong duality from Lemma \ref{lemma: zero gap} and the distributed dual framework \cite{uribe2021dual}, we derive the following result.

\begin{thm}
\label{thm: main}
Assume that Problem~\eqref{eq: qip} admits a unique solution $(\hat{s}^1, \dots, \hat{s}^N)$ and $(\hat{w}^1, \dots, \hat{w}^N)$. Furthermore, if at every iteration $t$ in Algorithm~\ref{algo: deslr}, the step size $\alpha_t$ is chosen to be square summable but not summable, i.e.,
\begin{equation}
    \notag
    \begin{split}
        \sum_{t = 1}^\infty \alpha_t & = + \infty, \quad \text{and} \quad 
        \sum_{t = 1}^\infty \alpha_t^2  < \infty,
    \end{split}
\end{equation}
then for every $\epsilon > 0$, there exist $T$ such that
\begin{equation}
    \label{eq: error}
    \frac{1}{\|E(\mathcal{G})\|} \sum_{(i,j) \in E(\mathcal{G})} \|w^i_T - w^j_T\| \leq \epsilon.
\end{equation}
Furthermore, we have that $\lim_{t \rightarrow \infty} w^i_T = \hat{w}$, where $\hat{w}$ is the optimal solution of \eqref{eq: qip}.
\end{thm}

\begin{proof}
Since we are maximizing the concave function $\phi(y)$ using square summable but not summable step sizes, by  \cite[Theorem 7]{nedic2018network}, we have $\lim_{t \rightarrow \infty} y_t = \hat{y}$, where $\hat{y}= \argmax_y \phi(y)$.
For the remaining of the proof, we only need to show that the sequence of optimal solutions $s^1_t, \dots, s^N_t$ and $w^1_t, \dots, w^N_t$ of $$\min_{\substack{-Ms^i_t \leq w^i \leq Ms^i_t, \\ \forall i \in \llbracket N \rrbracket}} f(w, y_t)$$ at each iteration will converge to the optimal sparse regressor. 

We refer to the variables $s^1, \dots, s^N$ as the sparsity variables as they control which indices of the regressors can be non-zero and the variables $w^1, \dots, w^N$ as the regression variables.
To prove the desired result, we first show that the sequence of optimal solutions of the sparsity variables converges.
Let 
$$g(s^1, \dots, s^N, y) = \min_{\substack{-Ms^i \leq w^i \leq Ms^i \\ \forall i \in \llbracket N \rrbracket}} f(w, y).$$
Then, it holds that
$$\phi(y) = \min_{s^i \in S^p_k \ \forall i \in \llbracket N \rrbracket} g(s^1, \dots, s^N, y).$$
By definition, we have that $(s^1_t, \dots, s^N_t)$ is an optimal solution of 
$$\min_{s^i \in S^p_k \forall i \in \llbracket N \rrbracket} g(s^1, \dots, s^N, y_t).$$
Since $(S^p_k)^N$ is a compact set, there exist a subsequence $\{s^1_{t_j}, \dots, s^N_{t_j}\}^{\infty}_{j = 1}$ of the sequence $\{s^1_t, \dots, s^N_t\}^{\infty}_{t = 1}$ such that $$\lim_{j \rightarrow \infty} (s^1_{t_j}, \dots, s^N_{t_j}) = (\bar{s}^1, \dots, \bar{s}^N) \in (S^p_k)^N.$$
Moreover, because $(S^p_k)^N$ is finite, there exist a positive number $K_1 > 0$, such that for every $j \geq K_1$, we have $(s^1_{t_j}, \dots, s^N_{t_j}) = (\bar{s}^1, \dots, \bar{s}^N)$.

Furthermore, for every $j \geq K_1$, it holds that $$g(s^1_{t_j}, \dots, s^N_{t_j}, y_t) = g(\bar{s}^1, \dots, \bar{s}^N, y_{t_j}) \leq g((\hat{s}^1, \dots, \hat{s}^N, y_{t_j}),$$ 
by the optimality of $\bar{s}^1, \dots, \bar{s}^N$.
Nevertheless, when $(\bar{s}^1, \dots, \bar{s}^N)$ is fixed, the function $g(\bar{s}^1, \dots, \bar{s}^N, y)$ is continuous with respect to the variables $y$, thus 
\begin{align*}
\lim_{j \rightarrow \infty} g(\bar{s}^1, \dots, \bar{s}^N, y_{t_j}) &= g(\bar{s}^1, \dots, \bar{s}^N, \hat{y}) \\
&\leq g(\hat{s}^1, \dots, \hat{s}^N, \hat{y}).
\end{align*}
However, by definition of $\hat{y}$, we also have $g(\bar{s}^1, \dots, \bar{s}^N, \hat{y}) \geq g(\hat{s}^1, \dots, \hat{s}^N, \hat{y})$.
Thus, by the uniqueness of optimal solution of \eqref{eq: qip}, we must have $(\bar{s}^1, \dots, \bar{s}^N) = (\hat{s}^1, \dots, \hat{s}^N)$.

We have shown that every convergent subsequence of $\{(s^1_t, \dots, s^N_t)\}$ converges to the same accumulation point $(\hat{s}^1, \dots, \hat{s}^N)$.
Therefore, there exist $K_2 > 0$ such that for every $t > K_2$, we have
$$(s^1_t, \dots, s^N_t) = (\hat{s}^1, \dots, \hat{s}^N),$$
and
$$(w^1_t, \dots, w^N_t) \in \argmin_{\substack{-M\hat{s}^i \leq w^i \leq M\hat{s}^i, \\ \forall i \in \llbracket N \rrbracket}} f(w, y_t),$$
which is a strictly convex quadratic optimization problem.
Hence, $w^1_t, \dots, w^N_t$ is unique.
Therefore, $$\lim_{t \rightarrow \infty} (w^1_t, \dots, w^N_t) = (\hat{w}^1, \dots, \hat{w}^N).$$
The conclusion of the theorem follows, because $\hat{w}^1, \dots, \hat{w}^N$ satisfies the coupling constraints, i.e., $\hat{w}^1 = \dots = \hat{w}^N$.
\end{proof}

The consensus error in~\eqref{eq: error} can be interpreted as the average difference between two adjacent agents.
When this error goes to zero, the agents reach a consensus on a sparse regressor.
Theorem \ref{thm: main} requires the existence of a unique sparse regressor solving \eqref{eq: sparse linear regression} for the convergence to zero of~\eqref{eq: error}, which is a rather weak assumption for this problem.


In the next section, we present an outer approximation algorithm~\cite{duran1986outer} for solving the inner problem in Step 7 of Algorithm \ref{algo: deslr}.

\section{Quadratic Integer Programming Local Solver}
\label{section: local solver}

In this section, we provide an algorithm for solving the local problem in step $7$ of Algorithm \ref{algo: deslr}.
In particular, at each iteration, we need to solve a quadratic integer programming problem, which is given as
\begin{equation}
\label{eq: inner problem}
\begin{split}
    c^* = \min_{\|w\| \leq k} \frac{1}{2}\|Y - X w\|^2_2 + \frac{1}{\gamma} \|w\|^2_2 + \langle D, w \rangle.
\end{split}
\end{equation}
In \eqref{eq: inner problem}, for simplicity of notations, we drop the superscript denoting agents and the number of iterations.
The case where $D = 0$ is, in fact, a sparse linear regression problem, which can be solved effectively using an outer approximation algorithm \cite{bertsimas2020sparse}.
Motivated by the success of solving sparse linear regression in a very high-dimensional regime, we provide an alternative transformation of the objective function of \eqref{eq: inner problem}, which is favorable for an outer approximation algorithm.
Initially, we have
\begin{equation}
    \notag
    \begin{split}
        \frac{1}{2}\|Y - Xw\|^2_2 + \frac{1}{2\gamma} \|w\|^2_2 & = 
        \frac{1}{2} w^T(\frac{I}{\gamma} + X^TX)w  \\
        & \ \ + Y^TXw + \frac{1}{2}Y^TY.
    \end{split}
\end{equation}
For $\gamma > 0$, we have $(\frac{I}{\gamma} + X^TX)$ is a positive-definite matrix.
Hence, there exists an invertible matrix $\bar{X} \in \mathds{R}^{p \times p}$ such that $(\frac{I}{\gamma} + X^TX) = \bar{X}^T\bar{X}$.
Since $\bar{X}$ is invertible, there exist $\bar{Y} \in \mathds{R}^p$ such that $\bar{Y}^T\bar{X} = Y^TX$.
Therefore,
\begin{equation}
    \notag
    \begin{split}
        \frac{1}{2}\|Y - Xw\|^2_2 + \frac{1}{\gamma} \|w\|^2_2 & = \frac{1}{2}w^T\bar{X}^T\bar{X}w + \bar{Y}^T\bar{X}w + \frac{1}{2}\bar{Y}^T\bar{Y} \\
        & \ \ + \frac{1}{2}(Y^TY -\bar{Y}^T\bar{Y}) + \frac{1}{2\gamma}\|w\|^2_2.
    \end{split}
\end{equation}
Thus, we can rewrite the optimization problem \eqref{eq: inner problem} as
\begin{equation}
    \label{eq: new inner prob}
    \begin{split}
        c^* = \min_{\|w\|_0 \leq k} \frac{1}{2}(\|\bar{Y} - \bar{X}w\|^2_2 + \frac{1}{\gamma} \|w\|^2_2) + d^T\bar{X}w + \text{const} ,
    \end{split}
\end{equation}
where $d = (\bar{X}^{-1})^TD$.
The constant term in \eqref{eq: new inner prob} equals to $\frac{1}{2}(Y^TY - \bar{Y}^T\bar{Y})$ and is dropped for simplicity.

\noindent
For a fixed $s \in S^p_k$, we define $c(s)$ as the optimal value of \eqref{eq: new inner prob} with additional constraints $-Ms \leq w \leq Ms$, i.e.,
\begin{equation}
\label{eq: fixed s}
    c(s) \coloneqq \min_{-Ms \leq w \leq Mw} \frac{1}{2} ||\bar{Y} - \bar{X}_sw||^2_2 + \frac{1}{2\gamma}||w||^2_2 + d^T\bar{X}_sw,
\end{equation}
where $\bar{X}_s = \bar{X}I_s$ and $I_s \in \mathds{R}^{p \times p}$ is the diagonal matrix whose diagonal is $s$.
Moreover, \eqref{eq: fixed s} becomes a (continuous) quadratic programming problem, $c(s)$ can be explicitly computed as:
\begin{equation}
\label{eq: value cs}
\begin{split}
    2c(s) = & -(\bar{Y}^T\bar{X}_s - d^T\bar{X}_s) \times (\frac{I}{\gamma} + X_s^TX_s)^{-1} \\
    & \ \ \ \times(\bar{X}_s^T\bar{Y} - \bar{X}_s^Td) + \bar{Y}^T \bar{Y}.
\end{split}
\end{equation}

In the next proposition, we provide a simpler equivalent transformation for $c(s)$, which enables a simple computation of the value of $c(s)$ and its gradient.

\begin{prop}
\label{prop: apply woodbury}
For a fixed $s \in S^p_k$, we have
$$c(s) = \frac{1}{2}(\bar{Y}^T - d^T)(I + \gamma \sum_{i=1}^p s_iK_i)^{-1}(\bar{Y} - d) - \frac{1}{2} d^Td + \bar{Y}^Td,$$
where $K_i = \bar{X}_i \bar{X}^T_i$ and $\bar{X}_i$ is the $i^{\text{th}}$ column of $\bar{X}$.
\end{prop}
\begin{proof}
We have,
\begin{equation}
    \notag
    \begin{split}
        c(s) & = {-}\frac{1}{2}(\bar{Y}^T {-} d^T)\bar{X}_s(\frac{I}{\gamma} {+} \bar{X}_s^T\bar{X}_s)^{-1}\bar{X}^T_s(\bar{Y} {-} d) {+} \frac{1}{2}\bar{Y}^T\bar{Y} \\
        & = \ \frac{1}{2}(\bar{Y}^T - d^T)(I - \bar{X}_s(\frac{I}{\gamma} + \bar{X}_s^T\bar{X}_s)^{-1}\bar{X}^T_s)(\bar{Y} - d) + \\
        & \ \ \ \ \ \frac{1}{2} \bar{Y}^T\bar{Y} - \frac{1}{2}(\bar{Y}-d)^T(\bar{Y}-d)\\
        & = \ \frac{1}{2}(\bar{Y}^T {-} d^T)(I {+} \gamma \bar{X}_s \bar{X}_s^T)^{-1}(\bar{Y} {-} d) {-} \frac{1}{2} d^Td {+} \bar{Y}^Td \\
        & = \ \frac{1}{2}(\bar{Y}^T - d^T)(I + \gamma \sum_{i=1}^p s_iK_i)^{-1}(\bar{Y} - d) \\
        & \ \ \ \ \ - \frac{1}{2} d^Td + \bar{Y}^Td.
    \end{split}
\end{equation}
In the second equality, we add and subtract $\frac{1}{2} (\bar{Y}-d)^T(\bar{Y} - d)$.
For the third equality, we apply the Woodbury matrix identity formula on $I - \bar{X}_s(\frac{I}{\gamma} + \bar{X}_s^T\bar{X}_s)^{-1}\bar{X}_s^T$.
Finally, since $s \in \{0,1\}^p$, we have $\bar{X}_s^T \bar{X}_s = \sum_{i = 1}^p s_iK_i$. \\
\end{proof}

By Proposition \ref{prop: apply woodbury}, the optimization problem \eqref{eq: inner problem} can now be reformulated as
\begin{equation}
\label{eq: binary opt}
    \min_{s \in S^p_k} \frac{1}{2}(Y^T {-} d^T)(I {+} \gamma \sum_{i=1}^p s_iK_i)^{-1}(Y {-} d) {-} \frac{1}{2} d^Td {+} Y^Td.
\end{equation}

The next proposition allows us to take the derivative of $c(s)$ for $s \in \text{conv}(S^p_k)$.

\begin{prop}
\label{prop: derivative of cs}
The function
$$c(s) = \frac{1}{2}(Y^T - d^T)(I + \gamma \sum_{i=1}^p s_iK_i)^{-1}(Y - d) - \frac{1}{2} d^Td + Y^Td$$
is convex and continuous on conv$(S^p_k)$. Furthermore, its gradient is given by
$$\nabla_{s} c(s) = -\frac{1}{2} \alpha(s)^T K_i \alpha(s),$$
where $\alpha(s) = (I + \gamma \sum_{i=1}^p s_iK_i)^{-1}(Y - d),$ for $i \in \llbracket p \rrbracket$.
\end{prop}

Proposition \ref{prop: apply woodbury} and Proposition \ref{prop: derivative of cs} derive a simple representation of $c(s)$ and its derivative.
These results help us to attain an outer approximation algorithm for solving \eqref{eq: new inner prob}, see Algorithm~\eqref{algo: local}.

\begin{algorithm}[tb!]
\caption{Outer Approximation for Solving Local Problem}
\label{algo: local}
\begin{algorithmic}[1]
\Procedure{Outer Approximation}{$(\bar{X}, \bar{Y})$, $\gamma$, $d$}
    \State $s_1 \leftarrow \text{ warm start }$
    \State $\eta_1 \leftarrow 0$
    \State $t \leftarrow 1$
    \While{$\eta_t < c(s_t)$}
        \State Compute $c(s_t)$ using Proposition \ref{prop: apply woodbury}
        \State Compute $\nabla c(s_t)$ using Proposition \ref{prop: derivative of cs}
        \State $s_{t+1}, \eta_{t+1} \leftarrow \argmin_{s, \eta}$ $$\{ \eta | s \in S^p_k, \ \eta \geq c(s_i) + \nabla c(s_i) (s - s_i) \ \forall i \in \llbracket t \rrbracket \}$$
        \State $t \leftarrow t + 1$
    \EndWhile
\State $\hat{s} \leftarrow s_t$
\State $\hat{d} \leftarrow (\bar{Y} -d)^T\bar{X}_s$
\State $\hat{w} \leftarrow \left( \frac{I_p}{\gamma} + \bar{X}_s^T\bar{X}_s) \right)^{-1} \hat{d} $
\EndProcedure
\end{algorithmic}
\end{algorithm}

According to \cite[Theorem 2]{fletcher1994solving}, Algorithm \ref{algo: local} will stop after a finite number of iterations and return the optimal solution of \eqref{eq: new inner prob}.
This implies that Problem~\eqref{eq: sparse linear regression} is dual-friendly~\cite{uribe2021dual, dunner2016primal, hiriart2004fundamentals, raginsky2012continuous}.
Certainly, to the best of our knowledge, a closed form or a polynomial algorithm does not exist for solving a quadratic integer programming problem.
However, the ability to yield the exact optimal solution of \eqref{eq: new inner prob} can help us derive a convergence analysis of Algorithm \eqref{algo: deslr} in terms of the number of iterations~$T$.

\section{Numerical Experiments}
\label{section: experiments}
To evaluate the convergence behavior of Algorithm \ref{algo: deslr}, we perform a series of experiments on different datasets and different network structures.
Before presenting our numerical experiments, we first describe how we generate our synthetic dataset.
The input data $X$ and its corresponding observations~$Y$ are generated following a linear relationship, i.e.,
$$Y = Xw^* + W,$$
where $w^*$ is considered to be the true regressor with $\|w\|_0 = k$.
To generate a true regressor $w^*$, we first pick $k$ indices from $\llbracket p \rrbracket$ as non-zero entries.
Afterwards, the chosen $k$ non-zero entries of $w^*$ are drawn from an uniform distribution on $[-1, 1]$.
The white noise $W$ is a random vector whose components are independently drawn from a normal distribution $\mathcal{N}(0, \sigma^2)$ for $\sigma>0$.
The input data $X = (x_1, \dots, x_n)$ is sampled from an independent identically distributed (i.i.d.) Gaussian distribution $\mathcal{N}(0, \Sigma)$, where $\Sigma_{i, j} = \rho^{|i - j|}$.
Finally, in all of the following computational experiments, we set $\delta = \rho = 0.1$.

It is well-known that the square summable but not summable step size, albeit guarantees convergence, does not provide the best numerical performance.
Hence, in all of the following experiments, initially, every agent starts with the same step-size $\alpha$ and a learning rate $\kappa$.
At each iterations $t$, after receiving regressor weights from their neighbors, each agent $i$ computes the average difference between its own local regressor and its neighbors', i.e.,
\begin{equation}
    \epsilon^t_i \coloneqq \frac{1}{\|N(i)\|} \sum_{j \in N(i)} \|w^i_t - w^j_t\|^2_2.
\end{equation}
If at a certain iteration $t$, every agent has their current local error larger than the previous step's $\epsilon_i^t \geq \epsilon_i^{t-1} \ \forall i \in \llbracket N \rrbracket$, we update the step size for all agent by damping it down using the pre-defined learning rate, $\alpha_{new} = \kappa \alpha$. 

In Figure \ref{fig: num features}, we show the simulation results of Algorithm~\ref{algo: deslr} on a small-world network of size $N = 50$ while we vary the number of features $p$ and the size of true support $k$.
The small-world property of our network is generated by using the Watts-Strogatz Algorithm \cite{watts1998collective} using a mean degree $K = 12$ and $\beta = 0.25$. 
For each different choice of $p$, we generate $5$ different datasets and plot the mean as a bold line and the $95\%$ confidence range of the error of \eqref{eq: error} across $T = 100$ iterations as a colored shadow.

\begin{figure}[tb!]
\centering
\includegraphics[width=0.45\textwidth]{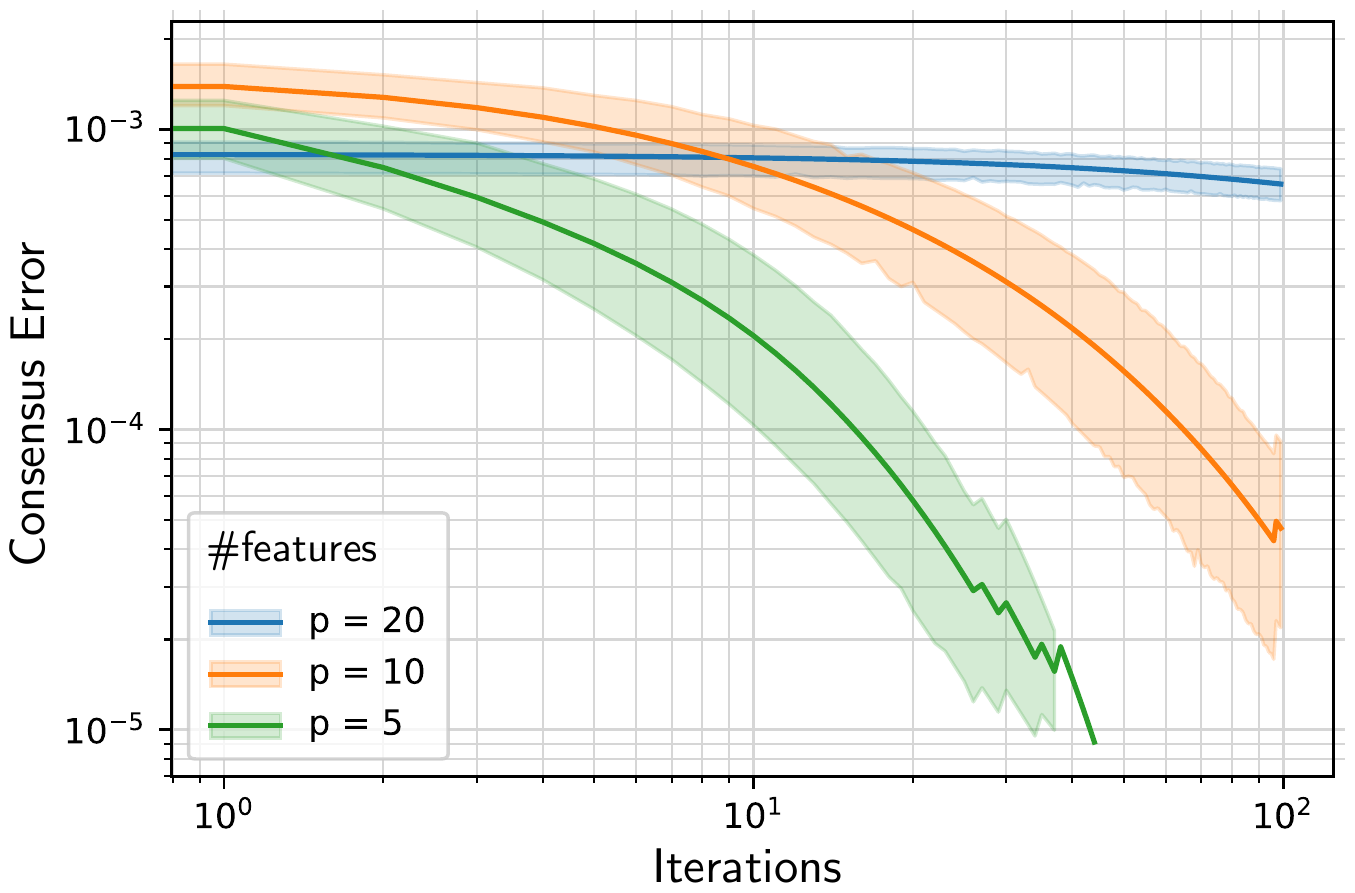}
\caption{\textbf{Convergence with respect to number of features $p$.} 
 The value for $k$ is chosen to be approximately $15\%$ the value of $p$. In particular, in the case where $p = 20$, $k = 3$, when $p = 10$, $k = 2$ and for $p = 5$, $k = 1$. The number of data-observations stored within each agent is set to scale linearly with $p$ and $k$, i.e., every agents has the same number of data points $n_i = 10pk$ for every $i \in \llbracket N \rrbracket$.}
\label{fig: num features}
\end{figure}

In Figure \ref{fig: network topology}, we analyze Algorithm \ref{algo: deslr} convergence behavior on different network classes.
We pick four common network structures: clique, star, cycle, and small-world.
Similarly, as in Figure~\ref{fig: num features}, for each network structure, we run Algorithm~\ref{algo: deslr} five times on $5$ different randomly generated dataset for $100$ iterations or until the error \eqref{eq: error} is below $10^{-5}$, and record the mean and the $95\%$ confidence interval of the error \eqref{eq: error}.

\begin{figure}[tb!]
\centering
\includegraphics[width=0.45\textwidth]{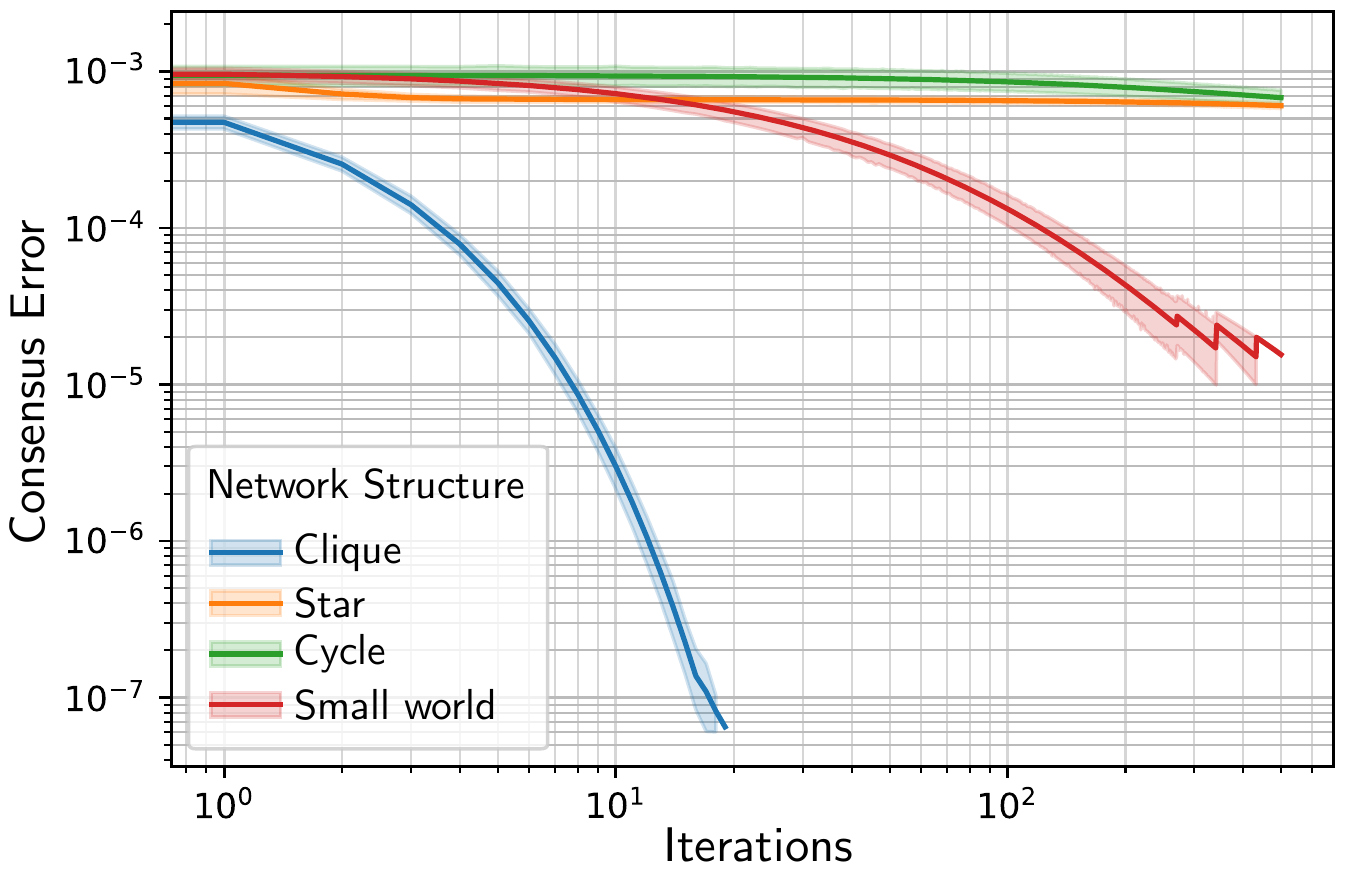}
\caption{\textbf{Influence of network structure on convergence rate.} The number of feature $p$ is set to be $18$, $k$ is set to be $3$, and the number of agents $N$ is chosen to be $50$ for every network structure.}
\label{fig: network topology}

\end{figure}
In Figure \ref{fig: network size}, we observe the error \eqref{eq: error} after $T = 100$ iterations on a path graph with various sizes. We first generate a dataset with $2000$ data points. Then, we distribute the generated dataset evenly among all agents for different network sizes. We then run Algorithm \ref{algo: deslr} on each of these settings. Figure~\ref{fig: network size} shows that as the number of agents increases, the number of iterations needed to reach the same error increases.

\begin{figure}[tb!]
\centering
\includegraphics[width=0.45\textwidth]{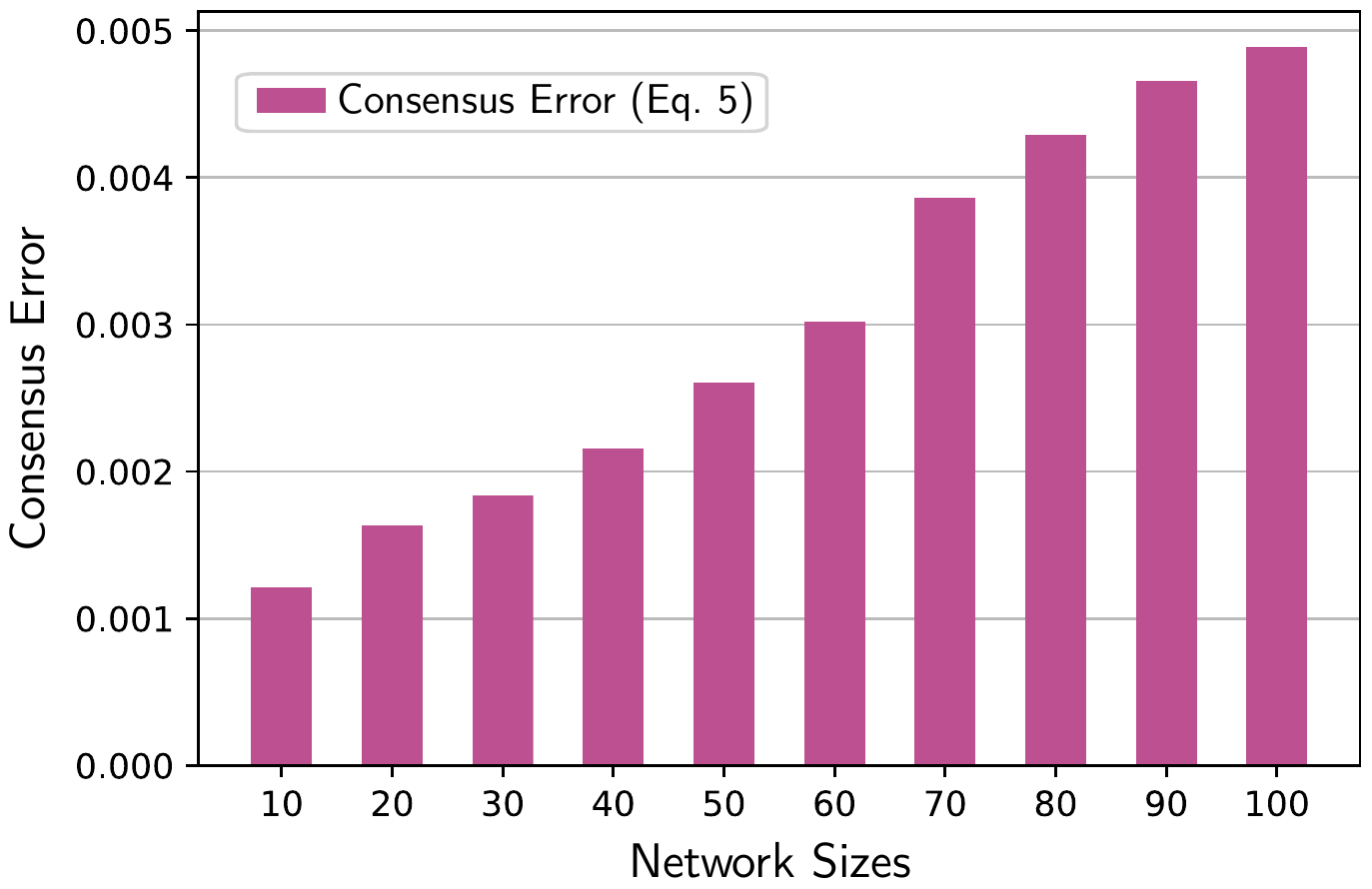}
\caption{\textbf{Influence of network's size on convergence rate.} The dataset is generated using a regressor with $p = 18$ features where there are $3$ non-zero coefficients, and the underlying graphs for the networks are paths.}
\label{fig: network size}
\end{figure}

\begin{rem}
Note that the optimal solution $\hat{w}$ of \eqref{eq: qip} is not the same the true regressor $w^*$.
Their support is the same, however, due to the $L_2$-regularization term, the value $\|\hat{w} - w^*\| > 0$.
Moreover, we have shown in Theorem \ref{thm: main} that, at each agent $i$, we have a Cauchy sequence $\{w^i_t\}_{t = 1}^\infty$ that converges to the optimal solution $\hat{w}$.
Therefore, the sub-optimal gap $\|w^i_t - \hat{w}\|$ behaves similarly to the consensus error shown in Figures~\ref{fig: num features}, \ref{fig: network topology}, and \ref{fig: network size}.
\end{rem}

For our current computational implementation, we are limited to solving the sparse linear regression problem with small $p$. This is because, at each iteration, we need to solve $N$ integer programming problems. This requires extensive computing resources.
Our software implementation for solving the inner QIP problems is not faster than a commercial solver like Gurobi \cite{gurobi}. Fast and efficient computational implementations of the proposed method are out of the scope of the present paper and are left for future work.

\section{Conclusion and Future Research}
\label{section: conclusions}
We have presented a decentralized scheme for solving exact sparse linear regression problems in the current work. We prove the convergence of the proposed method to the desired sparse solution.
Our main contribution sets a foundational approach for the study of distributed optimization methods for larger problem classes with explicit sparsity constraints.
The benefits of distributed sparse regression can be summarized as: interpretable regressors, solutions for distributed storage, and sparse communication between machines.

Future work should investigate the theoretical convergence rate of Algorithm~\ref{algo: deslr} knowing that the problem has a dual-friendly structure. Numerical experiments on larger problem scales should be studied.
As possible extensions, we consider the case where the graph is directed or time-varying as described in \cite{nedic2015nonasymptotic}.
Another direction is to study the convergence behavior of Algorithm~\ref{algo: deslr} when the sparse regressor at each agent is not optimal.
As we can see in Algorithm~\ref{algo: local}, the gap $\eta_t - c(s_t)$ is non-decreasing, and the moment this quantity attains a non-negative value, we are at an optimal solution.
Thus, $\eta_t - c(s_t)$ is a surrogate sub-optimality gap.
Furthermore, at any iteration of Algorithm~\ref{algo: local}, we have a sparse regressor.

\addtolength{\textheight}{-12cm}   





\bibliographystyle{IEEEtran} 
\bibliography{root}

\end{document}